\documentclass[10pt]{amsart}
\pdfoutput=1
\usepackage{amssymb}
\usepackage{graphicx}
\usepackage{tikz}
\usepackage{ifthen}

\usetikzlibrary{calc}
\usetikzlibrary{decorations.markings}
\usetikzlibrary{decorations.pathreplacing}

\theoremstyle{plain}
\newtheorem{thm}{Theorem}[section]
\newtheorem{lem}[thm]{Lemma}

\newtheorem{cor}[thm]{Corollary}

\theoremstyle{definition}
\newtheorem{defn}[thm]{Definition}
\newtheorem{nota}[thm]{Notation}

\everymath{\displaystyle}

\newcommand{\qbinom}[2]{\genfrac{[}{]}{0pt}{}{#1}{#2}}
\newcommand{\id}{\mathrm{id}}

\newcommand{\betaP}[3] {%
    \draw[->] (#1+0.2,#2+0.0)  arc(0:400:0.2) node[right]{$#3$};
}

\newcommand{\circles}[4] {%
    \draw[
        decoration={markings, mark=at position 0.25 with {\arrow{#4}}},
        postaction={decorate}
        ]
 (#1,#2) circle (0.2cm);
\draw (#1+.05,#2+0.3) node[anchor=west]  {$#3$};
\draw (#1-.05,#2-0.3) node  {\phantom{$#3$}};
}

\newcommand{\arc}[5] {%
   \draw[rounded corners=.25cm,
        decoration={markings, mark=at position 0.5 with {\arrow{#5}}},
        postaction={decorate}]
   (#3-0.3,#1) --(#3-0.3,#1-#2) --(#3+0.3,#1-#2)--(#3+0.3,#1) ;
\draw (#3+.4,#1-#2) node  {$#4$};
}

\title{The pop-switch planar algebra and the Jones-Wenzl idempotents}

\author{Ellie~Grano}
\address{Ellie Grano, Pepperdine University-NASC, 
24255 Pacific Coast Hwy, 
Malibu, CA 90263-4321}
\email{ellie.grano@pepperdine.edu}

\author{Stephen~Bigelow}
\address{Stephen Bigelow, Dept of Math, South Hall Room 6607, University of California, Santa Barbara, CA93106}
\email{bigelow@math.ucsb.edu}
\date{}

\begin{document}

\tikzset{mystyle/.style={baseline=(current bounding box), scale=0.8, every node/.style={transform shape}}}

\begin{abstract}

The Jones-Wenzl idempotents are elements of the Temperley-Lieb planar algebra
that are important, but complicated to write down.  
We will present a new planar algebra, the pop-switch planar algebra, which contains the Temperley-Lieb planar algebra.
It is motivated by Jones' idea of the graph planar algebra of type $A_n$.
In the tensor category of idempotents of the pop-switch planar algebra,
the $n$th Jones-Wenzl idempotent is isomorphic to a direct sum of $n+1$ diagrams consisting of only vertical strands.
\end{abstract}

\maketitle

\section{Introduction}
\label{sec:intro}

The Temperley-Lieb algebras were first introduced by Temperley and Lieb \cite{TL}
in their work on transfer matrices in statistical mechanics.  
Vaughn F. R. Jones independently rediscovered Temperley-Lieb algebras
in his work on von Neumann algebras \cite{Jones}.
He assembled these algebras together to form the Temperley-Lieb planar algebra,
the simplest example of a subfactor planar algebra. 

The Jones-Wenzl idempotents, first introduced in \cite{Wenzl}, are elements of the Temperley-Lieb algebras.
One way they arise naturally is in representation theory.
The Temperley-Lieb algebras encode
the category of representations of $U_q(\mathfrak{sl}_2)$,
and the Jones-Wenzl idempotents represent the irreducible representations.
Chapters in books have been devoted to them \cite{Kauffman}.
They have been categorified by \cite{JWcateg} and \cite{JWcateg2},
and generalized \cite{GenJW}.

While important, the Jones-Wenzl idempotents are difficult to write down explicitly.
The $n$th Jones-Wenzl idempotent  is a linear combination of
every diagram with $n$ non-intersecting strands.
The number of these diagrams is the $n$th Catalan number.
To find the coefficient of a given diagram requires a complicated algorithm
originally given by Frankel and Khovanov \cite{frenkel1997}
and later written down by Morrison \cite{Morr}.

In this paper,
we define the pop-switch planar algebra,
a new planar algebra that contains the Temperley-Lieb planar algebra.
Our original motivation was
a diagrammatic treatment of the graph planar algebra
introduced by Jones \cite{JonesGPA}. 
The pop-switch planar algebra captures  with simple diagrams the complicated calculations involved in working with objects in the graph planar algebra.

The main theorem of this paper shows that each Jones-Wenzl idempotent is isomorphic to a direct sum of diagrams with only vertical strands.  
It is to be hoped that
this makes them easier to work with,
and gives a new approach to some open problems.

\section{Background}

For convenience, we work over the field $\mathbb{C}$
and let $q$ be a nonzero complex number that is not a root of unity.
Many of the results hold over other fields,
but if $q$ is a root of unity
the proofs fail due to division by zero.

\begin{defn}
The $n$th {\em quantum number} is defined as 
$$[n] = [n]_q = \frac{q^n-q^{-n}}{q-q^{-1}}$$
and the {\em quantum binomial} is defined as
$$\qbinom{n}{k}=\frac{[n][n-1]\cdots[n-k+1]}{[k][k-1]\cdots[1]}$$
where $0 \le k \le n$ are natural numbers.
\end{defn}

We have the following identities.

\begin{lem}
\label{lem:q}
$[k+l]=[k][l+1] - [k-1][l]$.
\end{lem}

\begin{proof}
This follows from the definition and a simple computation.
\end{proof}

\begin{cor}
\label{cor:q}
$\qbinom{k+l}{l}=[l+1]\qbinom{k+l-1}{l}-[k-1]\qbinom{k+l-1}{l-1}$.
\end{cor}

\begin{proof}
After taking a common denominator and cancelling common terms,
this reduces to the previous lemma.
\end{proof}

\subsection{Planar algebras}

We won't define planar algebras in great detail.
See Jones' original paper \cite{Jones} for a formal definition.
See \cite{D2n} for a helpful introduction.

We will use what are sometimes called {\em vanilla} planar algebras.
These lack any of the optional extra features or properties
that are often included in the definition.

A planar tangle $T$ consists of:
\begin{itemize}
\item a disk $D$ called the {\em output disk},
\item a finite set of disjoint disks $D_i$ called the {\em input disks}
      in the interior of $D$,
\item a point called a {\em basepoint}
      of $\partial D$ and of each $\partial D_i$, and
\item a collection of disjoint curves called {\em strands} in $D$.
\end{itemize}
The strands can be closed curves,
or can have endpoints on $\partial D$
or $\partial D_i$ or both.
Apart from the endpoints,
the strands lie in the interior of $D$
and do not intersect $D_i$.
The basepoints do not coincide with endpoints of strands.
Planar tangles are considered up to isotopy in the plane.

It is sometimes possible to
insert a planar tangle $T_1$
into one of the input disks of another planar tangle $T_2$
to obtain a new planar tangle.
Specifically,
this is possible if
the number of endpoints on the output disk of $T_1$
is the same as the number of endpoints
on the chosen input disk of $T_2$.
Then we can use an isotopy to make the endpoints match up.
This still leaves an ambiguity of how to rotate $T_1$.
The basepoints remove this ambiguity:
we require the basepoint of the output disk of $T_1$
to coincide with the basepoint of the chosen input disk of $T_2$.

The planar tangles,
together with this operation of inserting one planar tangle
into an input disk of another,
form a rather general type of algebraic gadget called an
{\em operad}.
Briefly,
a planar algebra
is a representation of the operad of planar tangles.

More concretely,
a planar algebra $\mathcal{P}$ is a sequence of vector spaces 
$\mathcal{P}_i$ for $i \ge 0$.
Suppose $T$ is a planar tangle
with input disks $D_1, \dots, D_n$.
Let $d_i$ be the number of endpoints on $\partial D_i$
and let $d$ be the number of endpoints on $\partial D$.
Suppose $v_i \in \mathcal{P}_{d_i}$ for all $i$.
Then there is an action of $T$
$$T(v_1, \dots, v_n) \in \mathcal{P}_d.$$
The action of planar tangles
must be multilinear,
and it must be compatible with the operad structure
in a natural sense.

The definition of a planar algebra may seem complicated.
However it formalizes a fairly simple idea,
familiar to knot theorists,
of tangle-like diagrams
that can be glued together in arbitrary planar ways.
Perhaps the main novelty
is that we allow
formal linear combinations of diagrams,
which glue together in a multilinear way.

An example might help.

\subsection{The Temperley-Lieb planar algebra}

The simplest planar algebra is
the Temperley-Lieb planar algebra $\mathcal{TL}$.
The vector space $\mathcal{TL}_i$
is spanned by tangle diagrams that have no input disks
and $i$ endpoints on the output disk.

There is one relation.
A closed loop in a diagram
may be deleted
at the expense of multiplying by the scalar $q + q^{-1}$.
We call this the {\em bubble-bursting relation}.

If $i$ is odd then $\mathcal{T}_i$ is zero.
A basis for $\mathcal{T}_{2n}$
is given by tangle diagrams
that have $n$ strands and no closed loops.

In practice,
most planar algebras
can be thought of as formal linear combinations of diagrams
that are similar to Temperley-Lieb diagrams,
but with optional extra features,
like crossings, orientations, colors, or vertices.

\subsection{The category corresponding to a planar algebra}

Suppose $\mathcal{P}$ is a planar algebra.
We now describe how $\mathcal{P}$ can be thought of as a category.
In this context,
the input and output disks in the definition of $\mathcal{P}$
should be thought of as rectangles instead of round disks.

The category $\mathbf{C}$ corresponding to $\mathcal{P}$ is as follows.
\begin{itemize}
\item The objects are the non-negative integers.
\item The morphisms from $i$ to $j$
      are the elements of $\mathcal{P}_{i+j}$,
      thought of as having $i$ endpoints on the bottom of the rectangle
      and $j$ on the top.
\item The composition $f \circ g$
      is given by stacking $f$ on top of $g$.
\end{itemize}

Let $\mathcal{P}^j_i$ denote $\mathcal{P}_{i+j}$
with the elements treated as morphisms from $i$ to $j$.

An idempotent
is an element $p$ of $\mathcal{P}^n_n$
such that $p^2 = p$.

We can expand the objects in the category
by a construction known as the {\em Karoubi envelope}.
This new category $\mathbf{C}'$ is defined as follows.
\begin{itemize}
\item The objects of $\mathbf{C}'$ are
      the idempotents of $\mathbf{C}$.
\item The morphisms from $p$ to $q$
      are morphisms in $\mathbf{C}$ of the form $qxp$.
\end{itemize}

Next, note that
$\mathbf{C}$ and $\mathbf{C}'$
are also tensor categories,
where $x \otimes y$
is obtained by placing $x$ to the left of $y$.

Finally,
we can define a {\bf matrix category} of $\mathbf{C}'$.
The objects are formal direct sums of objects of $\mathbf{C}'$
and the morphisms are formal matrices.
Instead of this abstract definition,
all we need is the following lemma.

\begin{lem} \label{lem:dirSum}
Suppose $p$ and $q_1,\dots,q_n$ are idempotents such that
$$p = q_1 + \dots + q_n,$$
and $q_i q_j = 0$ whenever $i \neq j$.
Then 
$$p \simeq q_1 \oplus \dots \oplus q_n.$$
\end{lem}

\subsection{Jones-Wenzl idempotents}

The Jones-Wenzl idempotent $p_n$
is the unique element of $\mathcal{TL}^n_n$
such that
\begin{itemize}
\item $p_n \neq 0$
\item $p_n^2 = p_n$
\item $a p_n = 0$ if $a$ is any diagram
      that includes a strand with both endpoints at the bottom of the rectangle.
\item $p_n b = 0$ if $b$ is any diagram
      that includes a strand with both endpoints at the top of the rectangle.
\end{itemize}


Because of these last two properties, the Jones-Wenzl idempotents are sometimes referred to as ``uncappable."
If $q$ is a root of unity,
the Jones-Wenzl idempotents do not exist for all $n$.

\section{The pop-switch planar algebra}
\label{sec:pspa}

\subsection{The pop-switch planar algebra}

\begin{defn}\label{def:pspa}
Let the pop-switch planar algebra $\mathcal{PSPA}$ be the planar algebra generated by oriented strands modulo the following relations.
\begin{itemize}
\item The pop-switch relations
$$\begin{tikzpicture}[mystyle] 



\circles{-.65}{0}{}{<}

\draw[
        decoration={markings, mark=at position 1.0 with {\arrow{>}}},
        postaction={decorate}
        ]
        (-0.1,-0.9) --(-0.1,0.9); 
\draw[
        decoration={markings, mark=at position 1.0 with {\arrow{>}}},
        postaction={decorate}
        ]
        (0.4,0.9) --(0.4,-0.9); 

\end{tikzpicture} = \begin{tikzpicture}[mystyle] 

%

\arc{.9}{.4}{0}{}{>}
\arc{-.9}{-.4}{0}{}{<}

\end{tikzpicture}, \quad
\begin{tikzpicture}[mystyle] 



\circles{-.65}{0}{}{>}

\draw[
        decoration={markings, mark=at position 1.0 with {\arrow{>}}},
        postaction={decorate}
        ]
        (-0.1,-0.9) --(-0.1,0.9); 
\draw[
        decoration={markings, mark=at position 1.0 with {\arrow{>}}},
        postaction={decorate}
        ]
        (0.4,0.9) --(0.4,-0.9); 

\end{tikzpicture} = \begin{tikzpicture}[mystyle] 


%

\arc{.9}{.4}{0}{}{<}
\arc{-.9}{-.4}{0}{}{>}

\end{tikzpicture}$$
\item The bubble-bursting relation
$$\raisebox{-4pt}{\tikz \draw[->] (0,0) +(90:0.2) arc(90:460:0.2);} +
  \raisebox{-4pt}{\tikz \draw[->] (0,0) +(90:0.2) arc(460:90:0.2);} =
  (q + q^{-1}) \epsilon,$$
where $\epsilon$ denotes the empty diagram.
\end{itemize}
\end{defn}

This contains the Temperley-Lieb planar algebra; a non-oriented strand is the sum of each orientation.

We need some tools to move the diagrams around. 

Denote $n$ parallel strands oriented in the same direction
by a single oriented strand labelled $n$.
$$\begin{tikzpicture}[mystyle]

\draw[ decoration={markings, mark=at position 0.55 with {\arrow{>}}},
        postaction={decorate}]
   (0.0,-1.0)--(0.0,1.0) ;
\draw (.3,0.1) node  {$n$};

\end{tikzpicture}
 = \begin{tikzpicture}[baseline,scale=0.8, every node/.style={transform shape}]

\draw[ decoration={markings, mark=at position 0.7 with {\arrow{>}}},
        postaction={decorate}]
   (0.0,-1.0)--(0.0,1.0) ;
\draw[ decoration={markings, mark=at position 0.7 with {\arrow{>}}},
        postaction={decorate}]
   (-0.4,-1.0)--(-0.4,1.0) ;

 \node at (0.35,0) {...};
\draw[ decoration={markings, mark=at position 0.7 with {\arrow{>}}},
        postaction={decorate}]
   (0.7,-1.0)--(0.7,1.0) ;

\draw [decorate,decoration={brace,amplitude=.15cm},rotate=0] (-.6,1.1) -- (.9,1.1);
\draw (.15,1.5) node  {$n$};
\end{tikzpicture}.$$
If $n$ is a negative integer,
\(
\raisebox{2pt}{\begin{tikzpicture}[mystyle] \draw[->](0,-0.4) -- (0,0.4) node[inner sep=.1cm, below right]{$n$};\end{tikzpicture}}
= \,
\raisebox{2pt}{\begin{tikzpicture}[mystyle] \draw[->](0,0.4) -- (0,-0.4) node[inner sep=.1cm, above right]{$-n$};\end{tikzpicture}}
\)

Let $\iota_n$ denote $n$ vertical strands oriented up.
Let $\beta_n$ denote $n$ parallel strands
that form a bubble oriented counterclockwise.
Let $\alpha_n$ denote a $\beta_{-n}$ inside a $\beta_{n}$.
$$\iota_n = 
\raisebox{-8pt}{\tikz[scale=0.8, every node/.style={transform shape}] \draw[->](0,-0.4) -- (0,0.4) node[below right]{$n$};}
\quad 
\beta_n =
\raisebox{-4pt}{\tikz[scale=0.8, every node/.style={transform shape}] \draw[->] (0,0.7) +(10:0.2) arc(10:400:0.2) node[right]{$n$};}
\quad 
\alpha_n =
\raisebox{2pt}{\begin{tikzpicture}[mystyle]
\draw[->] (0.5,0.0)  arc(0:400:0.5) node[right]{$n$};
\draw[<-] (0.15,0.0)  arc(0:400:0.15) node[right]{$\scriptstyle{n}$};
\end{tikzpicture}}
.
$$

\begin{lem} 
\label{lem:teleport}
Suppose $x \in \mathcal{PSPA}_0$ 
and $y$ is a sequence of $2n$ vertical strands
such that $n$ are oriented up
and $n$ are oriented down.
Then $x \otimes y = y \otimes x$.
\end{lem}

\begin{proof}
Use the pop-switch relation repeatedly to create a gap
and pass $x$ through.
Then use the pop-switch relation repeatedly
to restore the original $2n$ vertical strands.
\end{proof}

\begin{lem} {\bf The multi-pop-switch relations}
The pop-switch relations hold for multiple strands.
\label{lem:multips}
$$\begin{tikzpicture}[mystyle] 

\draw[->] (-0.6,0)  arc(40:-320:0.2) node[right]{$n$};

\draw[
        decoration={markings, mark=at position 0.8 with {\arrow{>}}},
        postaction={decorate}
        ]
        (-0.0,-0.9) --(-0.0,0.9); 
\draw (0.25,0.45) node  {$n$};

\draw[
        decoration={markings, mark=at position 0.8 with {\arrow{>}}},
        postaction={decorate}
        ]
        (0.65,0.9) --(0.65,-0.9); 
\draw (0.9,-0.45) node  {$n$};

\end{tikzpicture} = \begin{tikzpicture}[mystyle] 

%

\arc{.9}{.4}{0}{n}{>}
\arc{-.9}{-.4}{0}{n}{<}

\end{tikzpicture},
\quad
\begin{tikzpicture}[mystyle] 


\draw[->] (-0.6,0)  arc(0:400:0.2) node[right]{$n$};

\draw[
        decoration={markings, mark=at position 0.2 with {\arrow{<}}},
        postaction={decorate}
        ]
        (-0.0,-0.9) --(-0.0,0.9); 
\draw (0.25,-0.45) node  {$n$};

\draw[
        decoration={markings, mark=at position 0.2 with {\arrow{<}}},
        postaction={decorate}
        ]
        (0.65,0.9) --(0.65,-0.9); 
\draw (0.9,0.45) node  {$n$};

\end{tikzpicture} = \begin{tikzpicture}[mystyle] 


%

\arc{.9}{.4}{0}{n}{<}
\arc{-.9}{-.4}{0}{n}{>}

\end{tikzpicture}$$
\end{lem}

\begin{proof} Without loss of generality, consider the first equality.
Induct on $n$.
The case $n = 1$ is the pop-switch relations.
For the case $n = k + 1$, move the innermost $\beta_{-k}$ across two strands using the previous lemma.  Then 
use the case $n = k$, and finally the case $n = 1$.
\end{proof}

\begin{cor}\label{cor:ia}
$\iota_k\otimes \alpha_n = \iota_k$ and $\iota_{-k}\otimes \alpha_{-n} = \iota_{-k}$ for $k\geq n\geq 0$.
\end{cor}

\begin{proof}
Consider $\iota_k\otimes \alpha_n$.
Use the multi-pop-switch relation by popping the innermost $\beta_n$ of the $\alpha_n$.
Then straighten out the $\iota_n$.
The other case is similar.
\end{proof}

\begin{cor}\label{cor:ab}
$$
\raisebox{2pt}{\begin{tikzpicture}[mystyle]
\draw[->] (0.5,0.0)  arc(0:400:0.5) node[right]{$n$};
\draw[->] (0.15,0.0)  arc(0:-680:0.15) ;
\end{tikzpicture}}=\alpha_{n}\otimes\beta_{n-1}
\quad\text{ and }\quad
\raisebox{2pt}{\begin{tikzpicture}[mystyle]
\draw[->] (0.5,0.0)  arc(0:-680:0.5) node[right]{$n$};
\draw[->] (0.15,0.0)  arc(0:400:0.15) ;
\end{tikzpicture}}=\alpha_{-n}\otimes\beta_{-n+1}.
$$
\end{cor}

\begin{proof}
Start with the left side of the first equality.
Use a multi-pop-switch relation on the $n-1$ strands, as shown below.
$$
\raisebox{2pt}{\begin{tikzpicture}[mystyle]
\draw[->] (0.5,0.0)  arc(0:400:0.5) node[right]{$n$};
\draw[->] (0.15,0.0)  arc(0:-680:0.15) ;
\end{tikzpicture}}=
\raisebox{2pt}{\begin{tikzpicture}[mystyle]
\draw[->] (1.5,0.6)  arc(30:40:1.2) node[right]{$n-1$};
\draw[] (1.5,0.6)  arc(30:330:1.2) ;
\draw[->] (3.5,0.2)  arc(30:40:0.5) node[right]{$n-1$};
\draw[decoration={markings, mark=at position 0.6 with {\arrow{>}}},
        postaction={decorate}]
(3.5,0.2)  arc(30:140:0.5)  arc(-40:-140:0.6) --(1.5,0.6);
\draw[decoration={markings, mark=at position 0.7 with {\arrow{<}}},
        postaction={decorate}] 
(3.5,0.2)  arc(30:-140:0.5)arc(40:145:0.67)--(1.5,-0.6) ;
\draw[->] (0.6,0)  arc(0:400:0.4);
\draw[->] (0.35,0)  arc(0:-680:0.15) ;
\end{tikzpicture}}
=
\raisebox{2pt}{\begin{tikzpicture}[mystyle]
\draw[->] (1.5,0.6)  arc(30:40:1.2) node[right]{$n-1$};
\draw[] (1.5,0.6)  arc(30:400:1.2) ;
\draw[->] (3.5,0.2)  arc(30:40:0.5) node[right]{$n-1$};
\draw[]
(3.5,0.2)  arc(30:400:0.5);
\draw[->] (0.4,0)  arc(0:400:0.4);
\draw[->] (0.15,0)  arc(0:-680:0.15) ;
\draw[->] (0.8,0.1)  arc(0:-680:0.15) node[right]{$\scriptstyle{n-1}$};
\end{tikzpicture}}$$
By Lemma \ref{lem:teleport} we can move the $\beta_{-n+1}$ into the $\alpha_1$ to achieve the result.
$$
=
\raisebox{2pt}{\begin{tikzpicture}[mystyle]
\draw[->] (0.5,0.0)  arc(0:400:0.5) node[right]{$n$};
\draw[<-] (0.15,0.0)  arc(0:400:0.15) node[right]{$\scriptstyle{n}$};
\draw[->] (1,0.0) +(10:0.2) arc(10:400:0.2) node[right]{$n-1$};
\end{tikzpicture}}
=\alpha_{n}\otimes\beta_{n-1}
$$
The second identity is proved similarly.
\end{proof}

\begin{lem} \label{lem:oio}
$\iota_n = \beta_{-n} \otimes \iota_n \otimes \beta_n$.
\end{lem}

\begin{proof}
This follows from the multi-pop switch relations.
$$\iota_n =
\begin{tikzpicture}[mystyle] \draw[->](0,-0.4) -- (0,0.4) node[inner sep=.1cm, below right]{$\scriptstyle{n}$};\end{tikzpicture} 
=
\begin{tikzpicture}[mystyle]
\draw[->] (0,-0.6) -- (0, -0.4) arc(180:90:.2) arc(270:450:.2) arc(270:180:.2) -- (0,0.6) node[right]{$\scriptstyle{n}$};
\end{tikzpicture}
=
\begin{tikzpicture}[mystyle]
\draw[->](0,-0.4) -- (0,0.4) node[inner sep=.1cm, below right]{$\scriptstyle{n}$};
\draw[->] (0.6,0)  arc(0:400:0.15) node[right]{$\scriptstyle{n}$};
\draw[->] (-0.4,0)  arc(0:-680:0.15) node[right]{$\scriptstyle{n}$};
\end{tikzpicture}
= \beta_{-n} \otimes \iota_n \otimes \beta_n
$$


\end{proof}

Now we give some relations involving the Jones-Wenzl idempotents $p_n$. 
First, we need some notation for them.
 We will use a rectangle to represent $p_n$.  It should always be assumed that $p_n \in P_n^n$ even if the strands are not drawn.

\begin{nota}
$p_n = \begin{tikzpicture}[mystyle]

\draw (-.4,-.8) -- (-.4,.8);
\draw (.4,-.8) -- (.4,.8);

\node at (0,.5) {$\ldots$};
\node at (0,-.5) {$\ldots$};

\draw [fill=white] (-0.5,-.3) rectangle (0.5,.3);

\node at (0,0) {$p_n$};

\end{tikzpicture}
 = \begin{tikzpicture}[mystyle]

\draw (0,-.8) -- (0,.8);

\node at (0.0,0.5) [inner sep=0.5,above right] {$\scriptstyle{n}$};

\node at (0,-.5) [inner sep=0.5,below right] {$\scriptstyle{n}$};

\draw [fill=white] (-0.5,-.3) rectangle (0.5,.3);

\end{tikzpicture}
=\begin{tikzpicture}[mystyle]

\draw (0,-0.8) -- (0,0);


\node at (0,-.5) [inner sep=0.5,below right] {$\scriptstyle{n}$};

\draw [fill=white] (-0.5,-.3) rectangle (0.5,.3);

\end{tikzpicture}
=\begin{tikzpicture}[mystyle]

\draw (0,0) -- (0,0.8);

\node at (0.0,0.5) [inner sep=0.5,above right] {$\scriptstyle{n}$};


\draw [fill=white] (-0.5,-.3) rectangle (0.5,.3);

\end{tikzpicture}
$
\end{nota}

We can make use of the the fact that they are uncappable.  

\begin{lem}
\label{lem:arcMove}
\begin{tikzpicture}[mystyle]
\draw [] (-.8,.5) rectangle (.8,0.7);
\draw (-0.2,-0.2) node  {$n$};
\draw[
        decoration={markings, mark=at position 0.8 with {\arrow{<}}},
        postaction={decorate}
        ]  (-.5,.5) -- (-.5, -.5);

\arc{.5}{.4}{.3}{}{>}
\end{tikzpicture}
$=(-1)^{n+1}$
\begin{tikzpicture}[mystyle] 
\draw [] (-.8,.5) rectangle (.8,0.7);
\draw (0.5,-0.2) node  {$n$};
\draw[
        decoration={markings, mark=at position 0.8 with {\arrow{<}}},
        postaction={decorate}
        ]  (.3,.5) -- (.3, -.5);

\arc{.5}{.4}{-.3}{}{<}
\betaP{.9}{0}{n}
\end{tikzpicture}
This relation remains true if all arrows are reversed.
\end{lem}

\begin{proof}
For the case $n = 0$,
use the fact that an unoriented cap gives zero.
For the case $n = 1$,
use the case $n = 0$ and the pop-switch relation.

For the general case,
use induction on $n$.  Start by using the case $n=k$ as follows:

\begin{align*}
\begin{tikzpicture}[mystyle]
\draw [] (-.9,.5) rectangle (.8,0.7);
\draw (-0.3,-0.2) node  {$\scriptstyle{k+1}$};
\draw[
        decoration={markings, mark=at position 0.8 with {\arrow{<}}},
        postaction={decorate}
        ]  (-.7,.5) -- (-.7, -.5);
\arc{.5}{.4}{.3}{}{>}
\end{tikzpicture}
&=
\begin{tikzpicture}[mystyle]
\draw [] (-.9,.5) rectangle (.8,0.7);
\draw (-0.3,-0.2) node  {$\scriptstyle{k}$};
\draw[
        decoration={markings, mark=at position 0.8 with {\arrow{<}}},
        postaction={decorate}
        ]  (-.5,.5) -- (-.5, -.5);
\arc{.5}{.4}{.3}{}{>}
\draw[
        decoration={markings, mark=at position 0.8 with {\arrow{<}}},
        postaction={decorate}
        ]  (-.8,.5) -- (-.8, -.5);
\end{tikzpicture}
=(-1)^{k+1}
\begin{tikzpicture}[mystyle] 
\draw [] (-.9,.5) rectangle (.8,0.7);
\draw[
        decoration={markings, mark=at position 0.8 with {\arrow{<}}},
        postaction={decorate}
        ]  (-.8,.5) -- (-.8, -.5);
\draw (0.5,-0.2) node  {$\scriptstyle{k}$};
\draw[
        decoration={markings, mark=at position 0.8 with {\arrow{<}}},
        postaction={decorate}
        ]  (.3,.5) -- (.3, -.5);
\arc{.5}{.4}{-.3}{}{<}
\betaP{.9}{0}{\scriptstyle{k}}
\end{tikzpicture}
 \intertext{Next use the case $n=1$, followed by Lemma \ref{lem:teleport}, to achieve the result.}
&=(-1)^{k+2}
\begin{tikzpicture}[mystyle] 
\draw [] (-.9,.5) rectangle (.8,0.7);
\draw[
        decoration={markings, mark=at position 0.8 with {\arrow{<}}},
        postaction={decorate}
        ]  (-.8,.5) -- (-.8, -.5);
\draw (0.5,-0.2) node  {$\scriptstyle{k}$};
\draw[
        decoration={markings, mark=at position 0.8 with {\arrow{<}}},
        postaction={decorate}
        ]  (.3,.5) -- (.3, -.5);
\arc{.5}{.4}{-.3}{}{>}
\betaP{.9}{0}{\scriptstyle{k}}
\end{tikzpicture}
=(-1)^{k+2}
\begin{tikzpicture}[mystyle] 
\draw [] (-.9,.5) rectangle (1.2,0.7);
\draw[
        decoration={markings, mark=at position 0.8 with {\arrow{<}}},
        postaction={decorate}
        ]  (0,.5) -- (0, -.5);
\draw (0.7,-0.2) node  {$\scriptstyle{k}$};
\draw[
        decoration={markings, mark=at position 0.8 with {\arrow{<}}},
        postaction={decorate}
        ]  (.5,.5) -- (.5, -.5);
\draw[<-] (0.2,0) +(400:0.1) arc(440:50:0.15);
\arc{.5}{.4}{-.5}{}{<}
\betaP{1.1}{0}{\scriptstyle{k}}
\end{tikzpicture}\\
&=(-1)^{k+2}
\begin{tikzpicture}[mystyle] 
\draw [] (-.9,.5) rectangle (.8,0.7);
\draw (0.7,-0.4) node  {$\scriptstyle{k+1}$};
\draw[
        decoration={markings, mark=at position 0.8 with {\arrow{<}}},
        postaction={decorate}
        ]  (.3,.5) -- (.3, -.5);
\arc{.5}{.4}{-.3}{}{<}
\betaP{1}{0.1}{\scriptstyle{k+1}}
\end{tikzpicture}
\end{align*}

\end{proof}

\section{Proof of the main theorem}
\label{sec:proof}

The aim of this section is to prove the following.

\begin{thm}
\label{thm:main}
The $n$th Jones-Wenzl idempotent is isomorphic to a direct sum of $n+1$ diagrams:

\begin{align*}
p_n &\simeq
\bigoplus_{i=0}^{n}  \iota_{-i} \otimes \iota_{n} \otimes \iota_{-i}\\
&=
\raisebox{2pt}{ \begin{tikzpicture}[mystyle]
    \draw[->](0,0) -- (0,1) node[inner sep=.1cm,below right]{$n$};
    \end{tikzpicture}
}
\oplus \,
\raisebox{2pt}{\begin{tikzpicture}[mystyle] 
    \draw[->](0,1) -- (0,0);
    \draw[->](0.5,0) -- (0.5,1) node[inner sep=0.1cm,below right]{$n$};
    \draw[->](1,1) -- (1,0);
    \end{tikzpicture}
}
\, \oplus \,
\dots
\oplus
\raisebox{2pt}{ \begin{tikzpicture}[mystyle]
    \draw[->](0,1) -- (0,0) node[inner sep=0.1cm,above right]{$n$};
    \draw[->](0.5,0) -- (0.5,1) node[inner sep=0.1cm,below right]{$n$};
    \draw[->](1,1) -- (1,0) node[inner sep=0.1cm,above right]{$n$};
    \end{tikzpicture}
}
\end{align*}
\end{thm}

\begin{proof}
Since $p_n$ is an idempotent, $p_n=p_n^2=p_n \id_n p_n$, where $\id_n$ is $n$ nonoriented parallel strands, the multiplicative identity in $\mathcal{TL}_n^n$.
Now write $\id_n$ as a sum of
$2^n$ different ways of orienting $n$ vertical strands.
Break this sum into $n+1$ sums depending on 
how many strands are oriented up.
 
\begin{defn}
Let $p^k_{n-k}$ denote the sum of $\binom{n}{k}$ diagrams
obtained from $p_n \id_n p_n$
by orienting $k$ strands up and $n-k$ strands down
in the $\id_n$.
\end{defn}
 
Then $p_n = p^0_n + p^1_{n-1} + \dots + p^n_0$.
If $k_1 \neq k_2$, then $p^{k_1}_{n-k_1}p^{k_2}_{n-k_2}=0$.
Thus, by Lemma \ref{lem:dirSum},
$$p_n \simeq p^0_n \oplus p^1_{n-1} \oplus \dots \oplus p^n_0.$$
It remains only to show
$$p^k_l \simeq \iota_{-l} \otimes \iota_{k+l} \otimes \iota_{-l}.$$
This is done in Lemma \ref{lem:isom}.
\end{proof}

To prove Lemma \ref{lem:isom},
we first define $X^k_l$,
which we will show is equal a scalar times $p^k_l$ in Lemma \ref{lem:pkl}.

\begin{defn} \label{defn:x}


$$X^k_l = \vcenter{\hbox{\includegraphics[scale=.275]{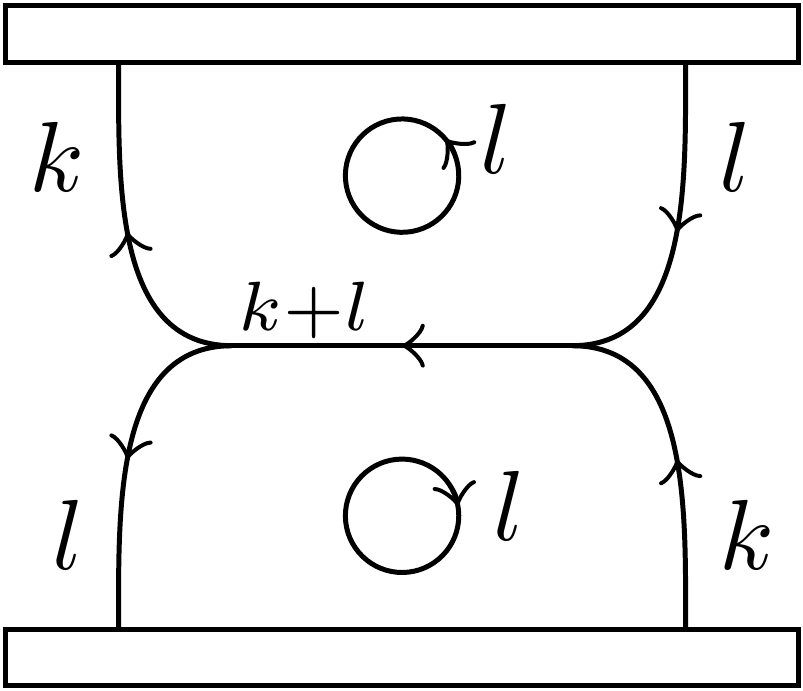}}}$$
\end{defn}

Lemmas \ref{lem:xkm1l} and \ref{lem:xklm1} are similar and begin the inductive step of the proof of Lemma \ref{lem:pkl}.

\begin{lem}\label{lem:xkm1l}
$p_{k+l} (X^{k-1}_l \otimes \iota_1) p_{k+l} = (-1)^{l} X^k_l \otimes \beta_{-l}.$
\end{lem}

\begin{proof}
 \begin{align*}
 p_{k+l} (X^{k-1}_l \otimes \iota_1) p_{k+l} &= \vcenter{\hbox{\includegraphics[scale=.275]{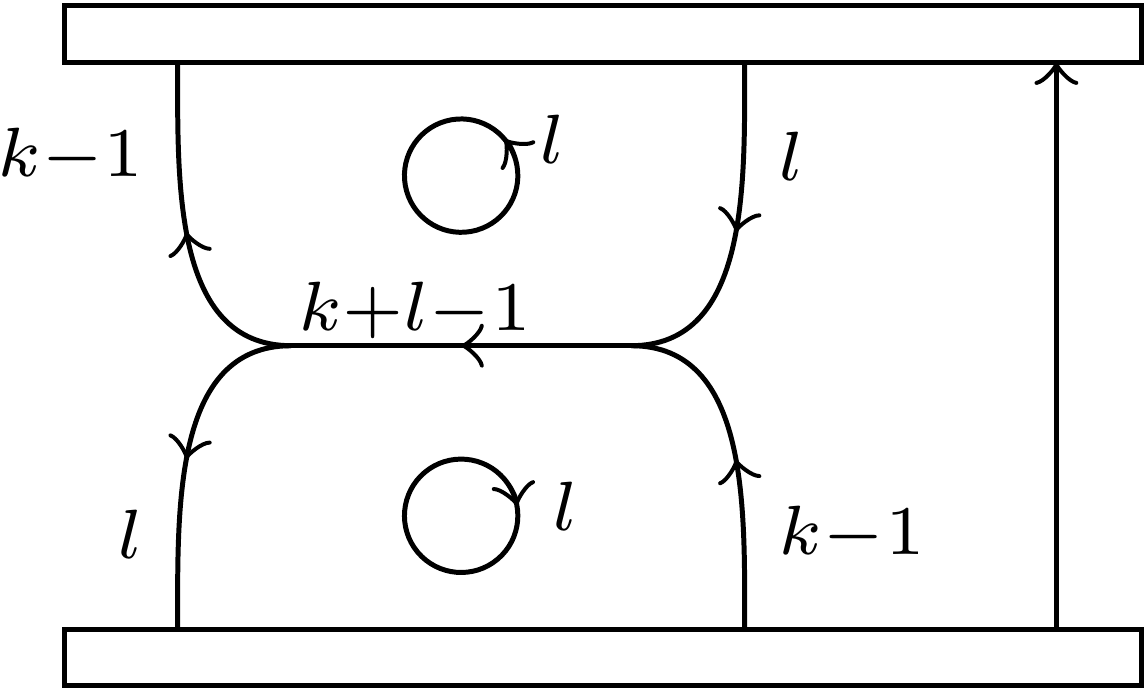}}}\\
 \intertext{By a pop-switch relation we have the following.}\\
 &= \vcenter{\hbox{\includegraphics[scale=.275]{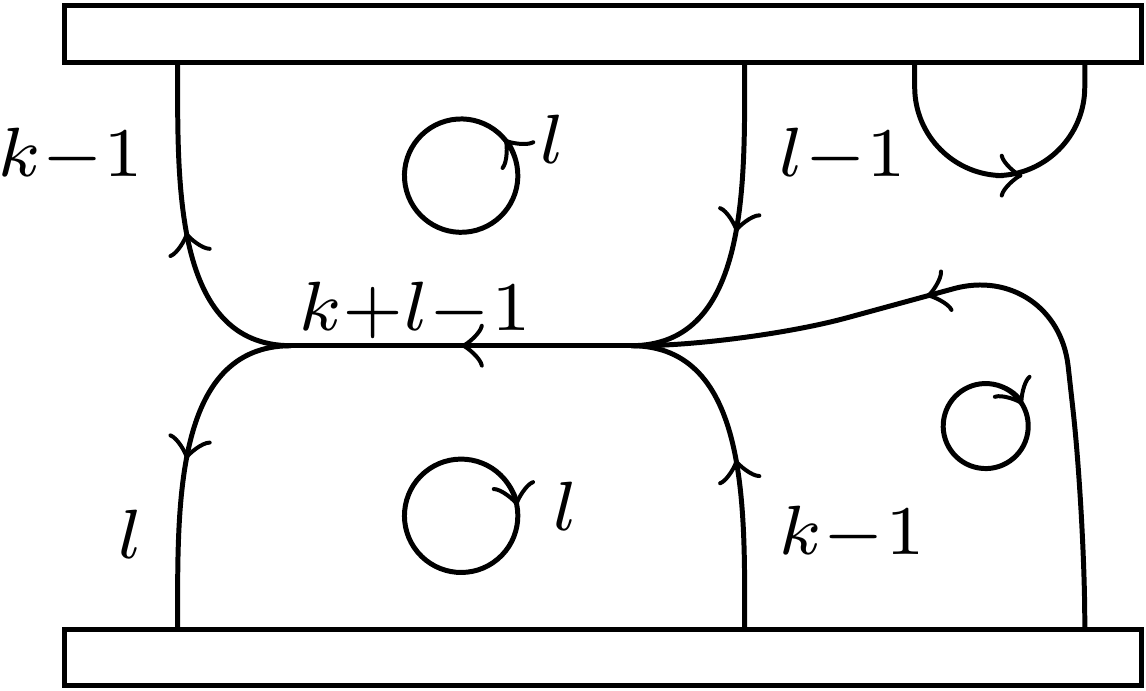}}}\\
 \intertext{Then by Lemma \ref{lem:arcMove} we can move the arc across the $l-1$ strands creating a $\beta_{l-1}$ on the right.  Next we use Lemma \ref{lem:oio} to replace the arc with $\beta_{-1}\otimes\iota_{1}\otimes\beta_{1}$.}\\
 &=(-1)^l \vcenter{\hbox{\includegraphics[scale=.275]{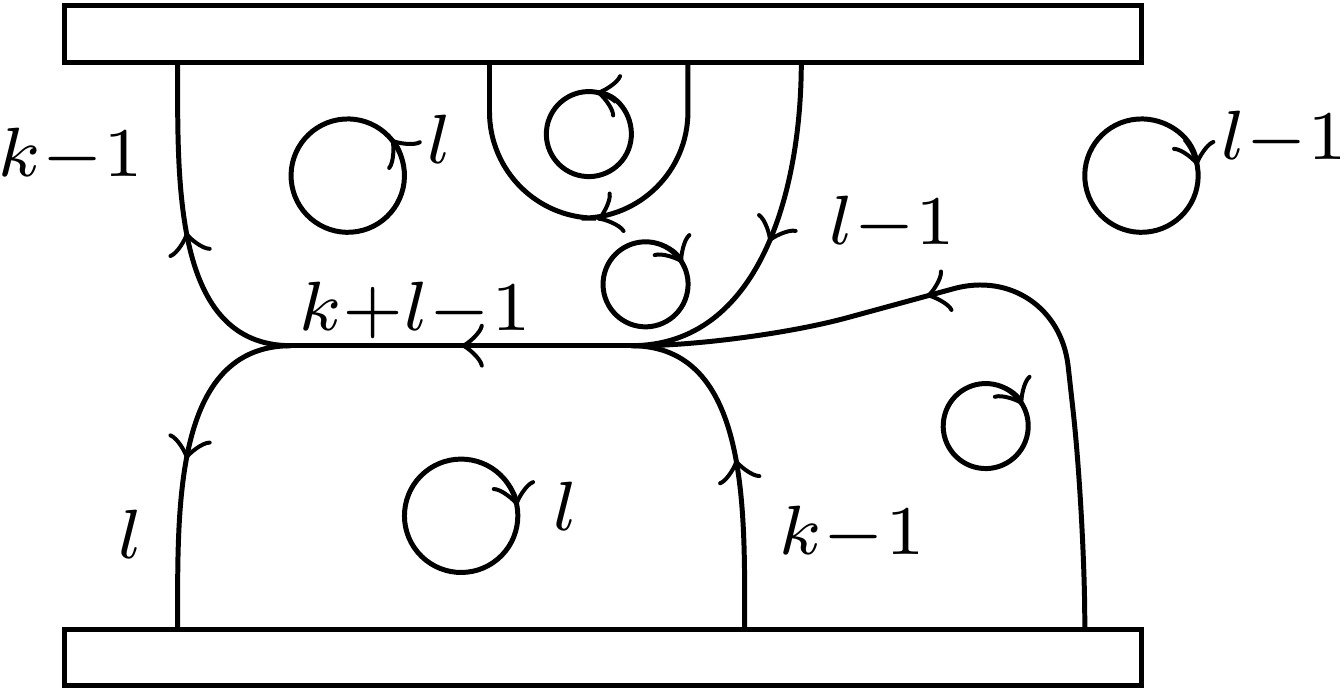}}}\\
 \intertext{Move the innermost $\beta_1$ from the $\beta_l$ to the far right across $l-1$ strands in both directions by Lemma \ref{lem:teleport}.}\\
 &=(-1)^l \vcenter{\hbox{\includegraphics[scale=.275]{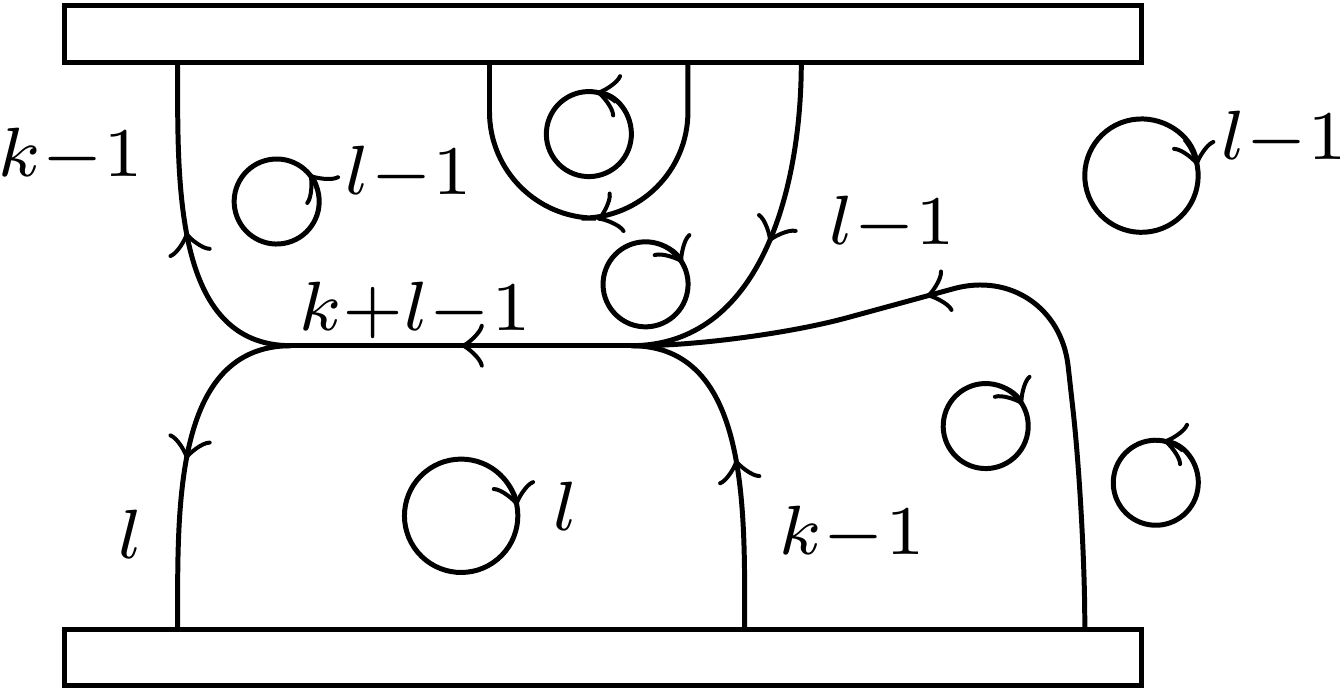}}}\\
  \intertext{Then remove the two bubbles on the bottom right of the diagram by Lemma \ref{lem:oio}.}\\
 &=(-1)^l\vcenter{\hbox{\includegraphics[scale=.275]{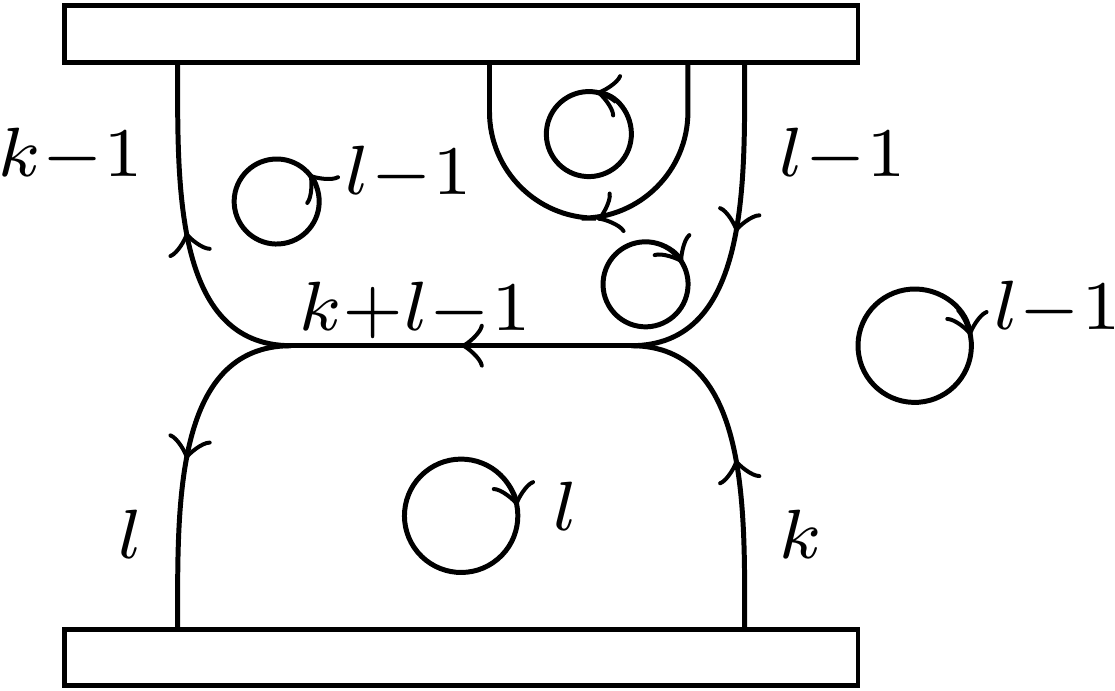}}}\\
   \intertext{Lastly, by Lemma \ref{lem:teleport} move the $\beta_{l-1}$ into the $\beta_{1}$ and the $\beta_{-1}$ into the $\beta_{-(l-1)}$.}\\
 &=(-1)^l \vcenter{\hbox{\includegraphics[scale=.275]{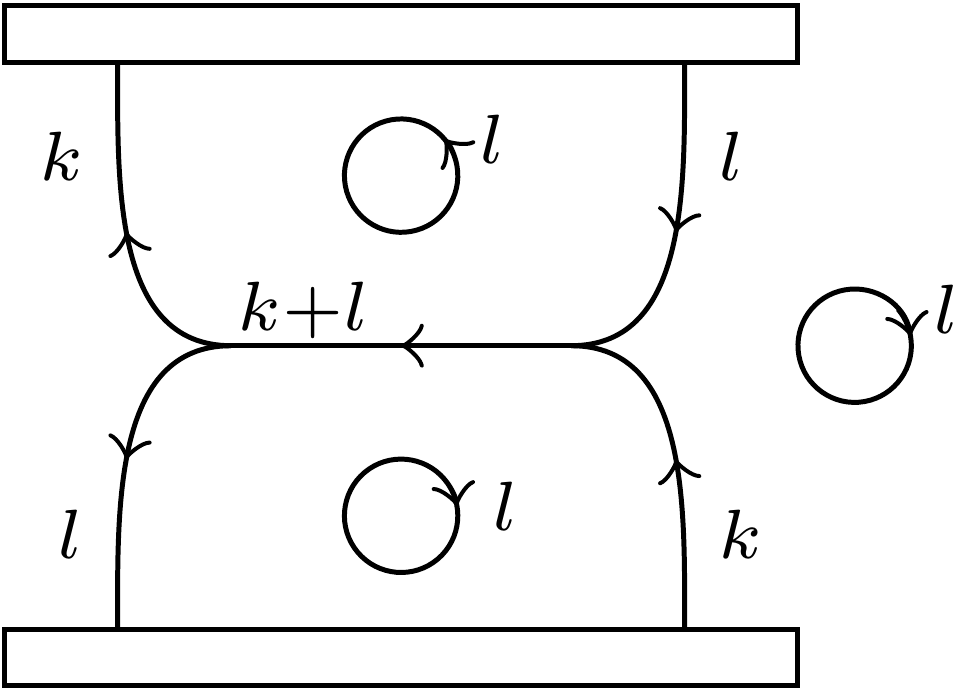}}}\\[.3cm]
 &=(-1)^{l} X^k_l \otimes \beta_{-l}\\
 \end{align*}
\end{proof}

\begin{lem}\label{lem:xklm1}
$p_{k+l} (X^k_{l-1} \otimes \iota_{-1}) p_{k+l} =
  (-1)^{k} X^k_l \otimes \beta_k.$
\end{lem}

\begin{proof}
 \begin{align*}
 p_{k+l} (X^{k}_{l-1} \otimes \iota_{-1}) p_{k+l} &=\vcenter{\hbox{\includegraphics[scale=.275]{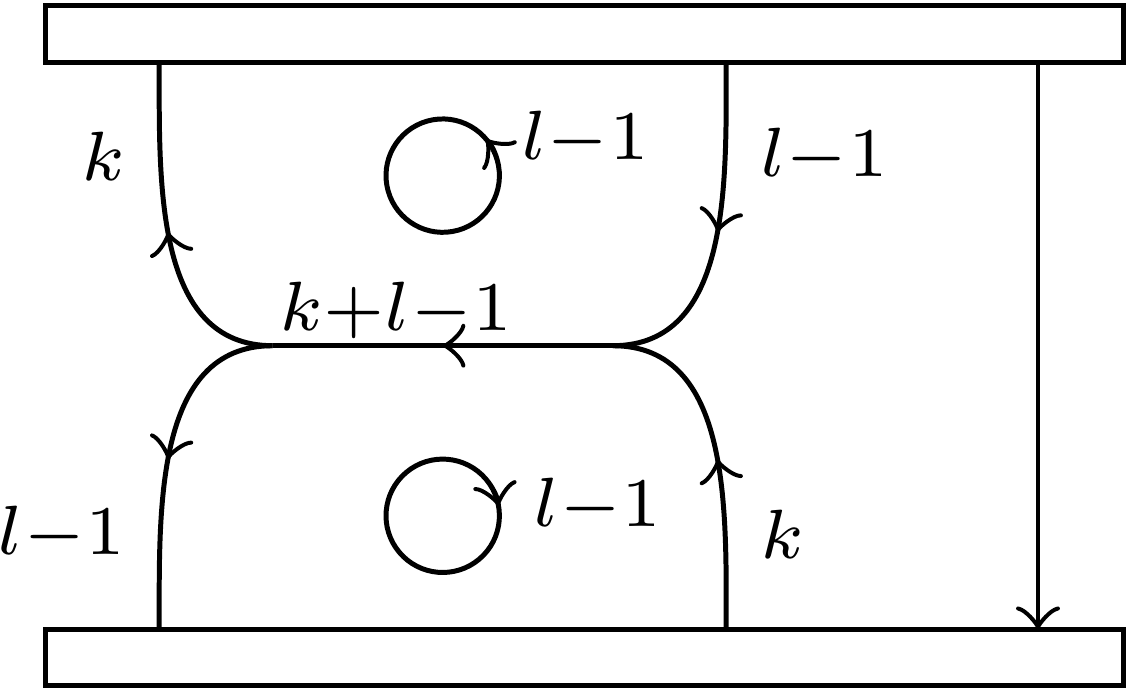}}}\\
 \intertext{By a pop-switch relation we have the following.}\\
 &=\vcenter{\hbox{\includegraphics[scale=.275]{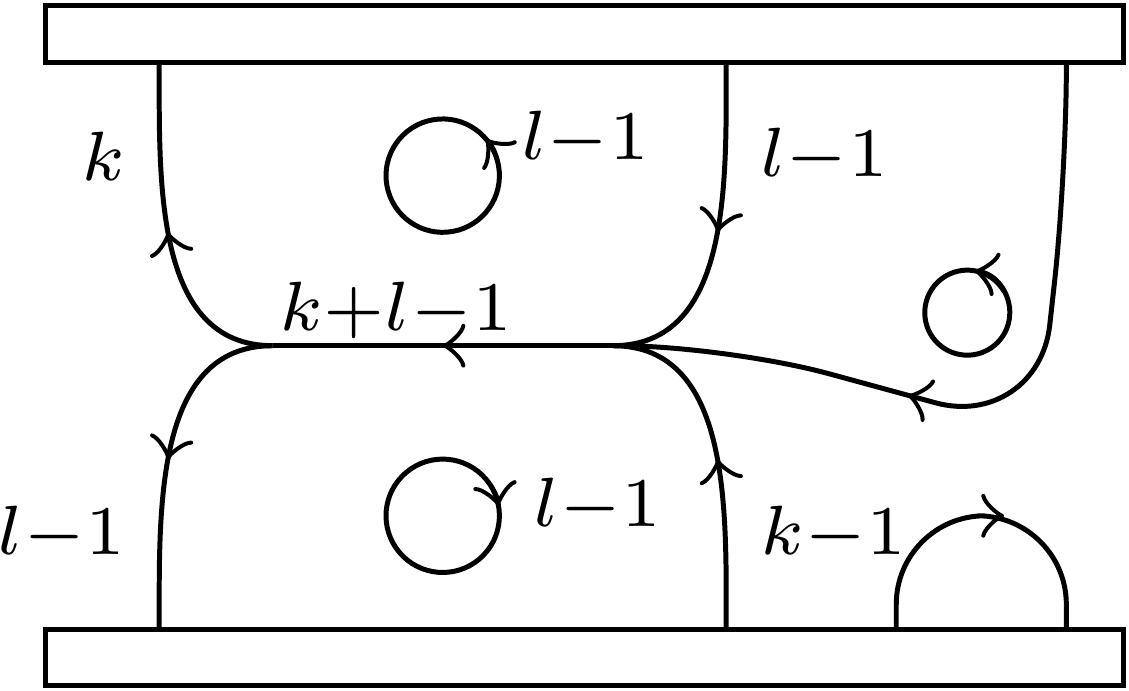}}}\\
 \intertext{Then by Lemma \ref{lem:arcMove} we can move the arc across the $k-1$ strands creating a $\beta_{-(k-1)}$ on the right.  Next we use Lemma \ref{lem:teleport} to move the $\beta_1$ across the $l-1$ strands in both directions into the $\beta_{l-1}$.}\\
 &=(-1)^k \vcenter{\hbox{\includegraphics[scale=.275]{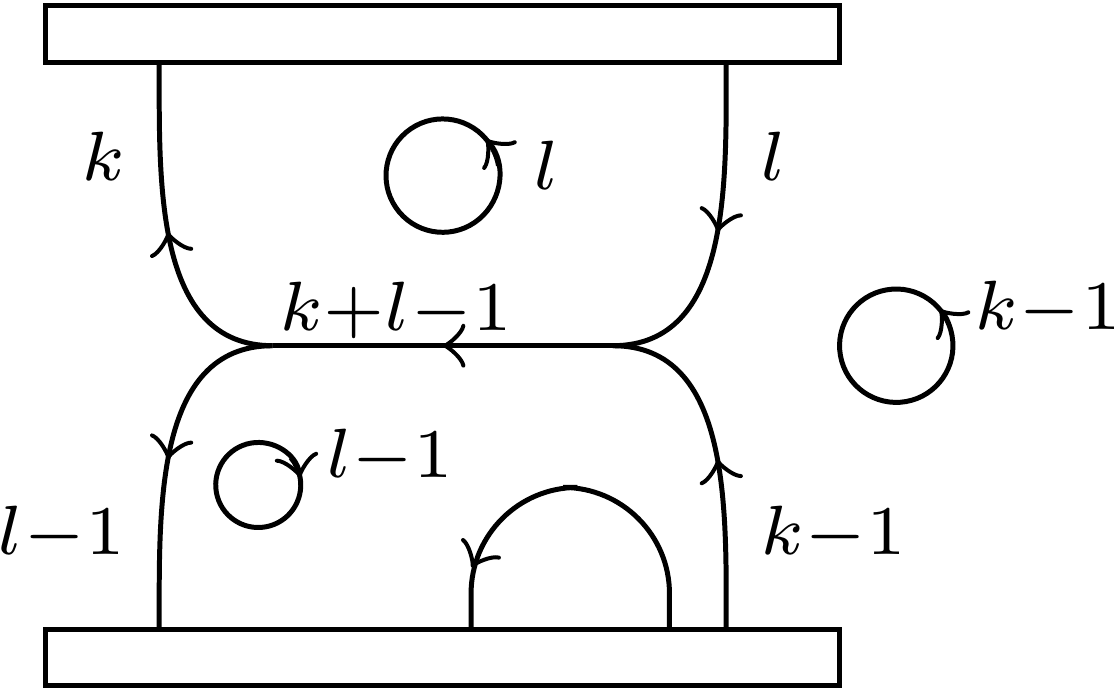}}}\\
 \intertext{Replace the arc with $\beta_{-1}\otimes\iota_{1}\otimes\beta_{1}$.}\\
 &=(-1)^k \vcenter{\hbox{\includegraphics[scale=.275]{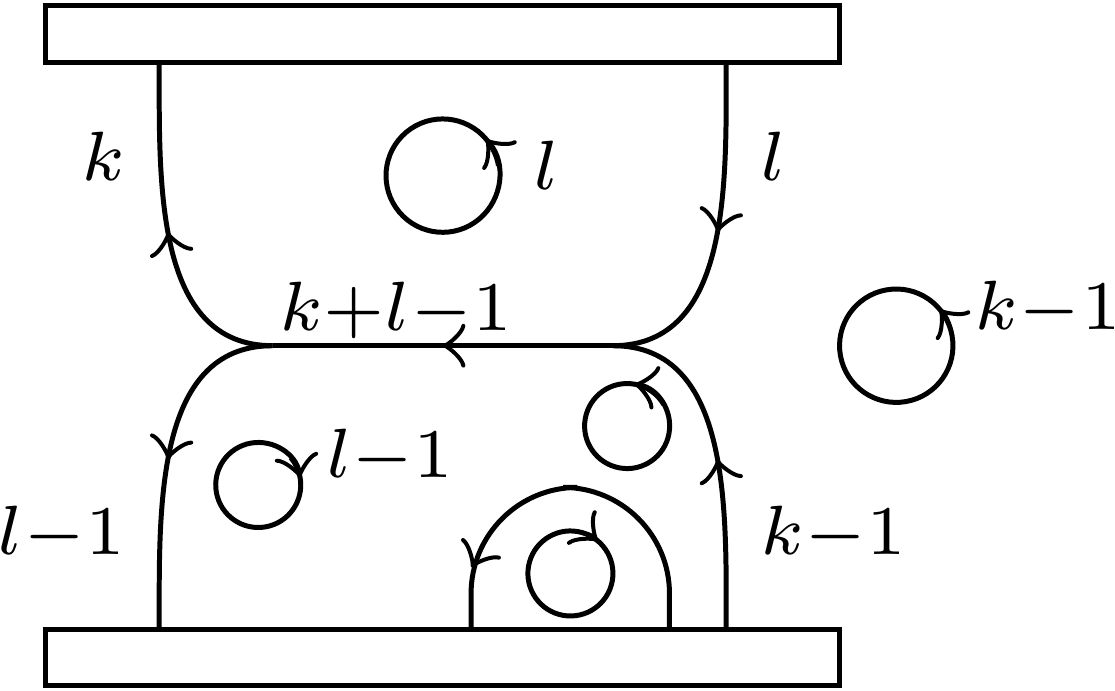}}}\\
   \intertext{Lastly, by Lemma \ref{lem:teleport} move the $\beta_1$ across the $k-1$ strands in both directions into the $\beta_{k-1}$.  By the same lemma, move the $\beta_{-(l-1)}$ into the $\beta_{-1}$.}\\
 &=(-1)^k \vcenter{\hbox{\includegraphics[scale=.275]{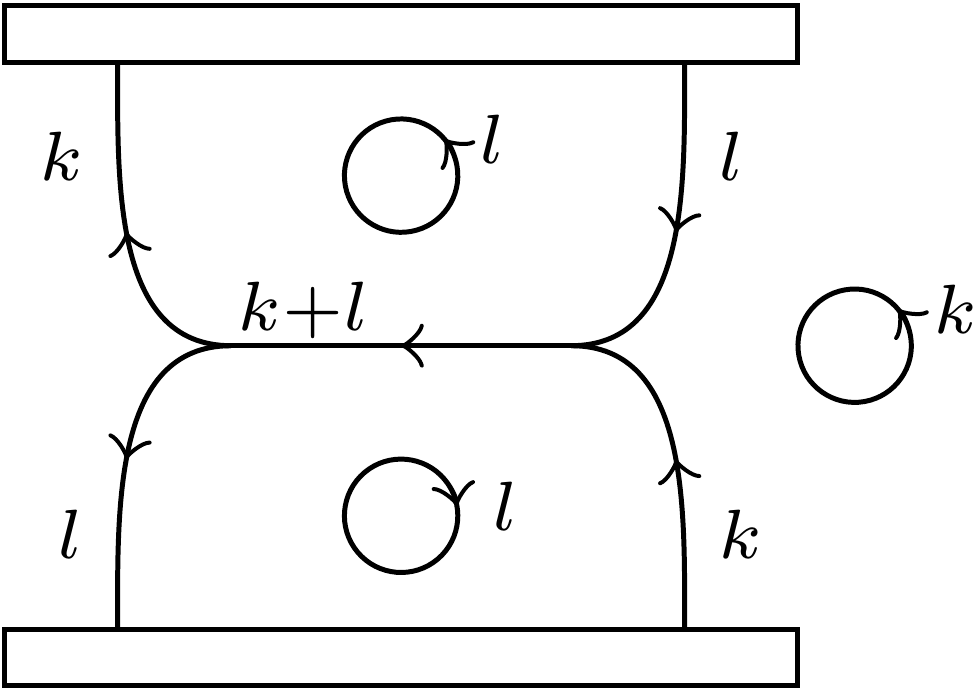}}}\\
 &=(-1)^{k} X^k_l \otimes \beta_{k}\\
 \end{align*}
\end{proof}

Lemma \ref{lem:bubbs} is the key to proving Lemma \ref{lem:quantPT}, which is required to complete the proof of Lemma \ref{lem:pkl}.
  It is worth noting that  in Lemma \ref{lem:quantPT} the $X^k_l$ merely acts as a catalyst to provide enough strands to use \ref{lem:bubbs}. 
All that is necessary is the presence of $\iota_{-l+1}$ and $\iota_{k-1}$ on the left as specified in Lemma \ref{lem:bubbs} for the purpose of implementing Corollary \ref{cor:ia}.  

\begin{lem}\label{lem:bubbs}
For $k\geq n-1$,
$$\iota_k\otimes \beta_n = [n]\iota_k\otimes\beta_1-[n-1]\iota_k$$
and
$$\iota_{-k}\otimes \beta_{-n} = [n]\iota_{-k}\otimes\beta_{-1}-[n-1]\iota_{-k}$$
\end{lem}

\begin{proof}
We prove the first identity, since the second is similar.
Consider the case $n=2$ with $k\geq 1$.
Use the bubble-bursting relation on the innermost loop of $\beta_2$.
Corollary \ref{cor:ia} then gives the result.
$$\iota_k\otimes \beta_2 = [2] \iota_k\otimes \beta_1 - \iota_k\otimes \alpha_1 = [2] \iota_k\otimes \beta_1 - \iota_k$$

Now assume $k\geq n-1$.
Use the bubble-bursting relation on the innermost loop of $\beta_n$.
Corollary \ref{cor:ab}, Corollary \ref{cor:ia}, and induction give
\begin{align*}
\iota_k\otimes \beta_n &= [2] \iota_k\otimes \beta_{n-1} - \iota_k\otimes \beta_{n-2} \\
&=[2]([n-1]\iota_k\otimes\beta_1 - [n-2] \iota_k)-([n-2]\iota_k\otimes\beta_1-[n-3]\iota_k)\\
&=([2][n-1]-[n-2])\iota_k\otimes \beta_1-([2][n-]-[n-3])\iota_k\\
&=[n]\iota_k\otimes\beta_1-[n-1]\iota_k
\end{align*}
\end{proof}


\begin{lem}\label{lem:quantPT}
If $k + l = n$ then
$$ \qbinom{n-1}{l}X^k_l\otimes\beta_{-l} + \qbinom{n-1}{k} X^k_l\otimes\beta_k = \qbinom{n}{k}X^k_l.$$
\end{lem}

\begin{proof}
Note that every term in the equation contains $X^k_l$.
However, the result will hold so long as there are both a $\iota_{-l+1}$ and $\iota_{k-1}$ on the left of each diagram in order to use Lemma \ref{lem:bubbs}.
Thus it suffices to prove
$$ \qbinom{n-1}{l}([l] \beta_{-1} - [l-1]) + \qbinom{n-1}{k}([k] \beta_1 - [k-1]) = \qbinom{n}{k}. $$
Use the identity
$$\qbinom{n-1}{l} [l] = \qbinom{n-1}{k} [k],$$
and the bubble bursting relation $\beta_{-1} + \beta_1 = [2]$
to eliminate $\beta_{-1}$ and $\beta_1$ from the left side.
Then simplify further using the identity $[2][l] - [l-1] = [l+1]$.
We obtain
$$\qbinom{n-1}{l}[l+1] - \qbinom{n-1}{k}[k-1].$$
By Corollary \ref{cor:q},
this is equal to $\qbinom{n}{k}$, as desired.
\end{proof}

\begin{lem} \label{lem:pkl}
$p^k_l = (-1)^{kl} \qbinom{k+l}{k} X^k_l$
\end{lem}

\begin{proof}
Induct on $n=k+l$.  Notice $p^1_0=\iota_1=X^1_0$ and $p^0_1=\iota_{-1}=X^0_1$.  Assume $k>0$ and $l>0$. Then 

\begin{align*}
p^k_l&=p_{k+l}(p^{k-1}_l\otimes\iota_1)p_{k+l}+p_{k+l}(p^{k}_{l-1}\otimes\iota_{-1})p_{k+l}
\intertext{By Lemma \ref{lem:xkm1l} and Lemma \ref{lem:xklm1},}
&=(-1)^{kl}\qbinom{k+l-1}{l}X^k_l\otimes\beta_{-l}+(-1)^{kl}\qbinom{k+l-1}{k}X^k_l\otimes\beta_{k}
\intertext{By Lemma \ref{lem:quantPT},}
&=(-1)^{kl}\qbinom{k+l}{k}X^k_l
\end{align*}

\end{proof}

\begin{lem}\label{lem:isom}
$p^k_l \simeq \iota_{-l} \otimes \iota_{k+l} \otimes \iota_{-l}$.
\end{lem}
\begin{proof}
The explicit isomorphisms are:
$$
f= (-1)^{kl}\qbinom{k+l}{k} \begin{tikzpicture}[mystyle]

\draw [] (-1.4,0.1) rectangle (1.4,-0.1);




\begin{scope}[decoration={markings, mark=at position 0.5 with {\arrow{>}}}]
 \draw[postaction={decorate}]
    (1, -.1) to[out=270, in=90] node[near start, auto]{$\scriptstyle{l}$} (0.6,-1.1);
 \draw[postaction={decorate}]
    (-0.6,-1.1) to[out=90, in=270] node[near end, auto]{$\scriptstyle{k}$} (-1, -.1);
\end{scope}

\draw[<-] (0,-0.5) +(400:0.2) arc(400:10:0.2) node[right]{$\scriptstyle{l}$};

\draw[->] (-1.4,-1.1)  arc(0:90:0.3) node[above]{$\scriptstyle{l}$};
\draw[] (-1.4,-1.1)  arc(0:180:0.3);

\end{tikzpicture},
\qquad
g=  \begin{tikzpicture}[mystyle]

\draw [] (-1.4,0.1) rectangle (1.4,-0.1);

\begin{scope}[decoration={markings, mark=at position 0.5 with {\arrow{>}}}]
  \draw[postaction={decorate}]
    (1, 0.1) to[out=90, in=-90] node[near start, auto, swap]{$\scriptstyle{k}$} (0.6,1.1);
  \draw[postaction={decorate}]
    (-0.6,1.1) to[out=-90, in=90] node[near end, auto, swap]{$\scriptstyle{l}$} (-1, 0.1);
\end{scope}

\draw[->] (0,0.5) +(400:0.2) arc(400:10:0.2) node[right]{$\scriptstyle{l}$};

\draw[->] (2,1.1)  arc(0:-90:0.3) node[below]{$\scriptstyle{l}$};
\draw[] (2,1.1)  arc(0:-180:0.3);




\end{tikzpicture}.
$$

Then
$f\circ g= (-1)^{kl} \qbinom{k+l}{k} X^k_l=p^k_l$ by Lemma \ref{lem:pkl}.  Thus $f\circ g$ is the identity morphism from $p^k_l$ to $p^k_l$.

On the other hand, $g\circ f=\iota_{-l} \otimes \iota_{k+l} \otimes \iota_{-l}$, the identity morphism from $\iota_{-l} \otimes \iota_{k+l} \otimes \iota_{-l}$ to $\iota_{-l} \otimes \iota_{k+l} \otimes \iota_{-l}$.

\begin{align*}
g\circ f&= (-1)^{kl}\qbinom{k+l}{k} \begin{tikzpicture}[mystyle]

\draw [] (-1.4,0.1) rectangle (1.4,-0.1);

\begin{scope}[decoration={markings, mark=at position 0.5 with {\arrow{>}}}]
  \draw[postaction={decorate}]
    (1, 0.1) to[out=90, in=-90] node[near start, auto, swap]{$\scriptstyle{k}$} (0.6,1.1);
  \draw[postaction={decorate}]
    (-0.6,1.1) to[out=-90, in=90] node[near end, auto, swap]{$\scriptstyle{l}$} (-1, 0.1);
\end{scope}

\draw[->] (0,0.5) +(400:0.2) arc(400:10:0.2) node[right]{$\scriptstyle{l}$};

\draw[->] (2,1.1)  arc(0:-90:0.3) node[below]{$\scriptstyle{l}$};
\draw[] (2,1.1)  arc(0:-180:0.3);

\begin{scope}[decoration={markings, mark=at position 0.5 with {\arrow{>}}}]
  \draw[postaction={decorate}]
    (1, -.1) to[out=270, in=90] node[near start, auto]{$\scriptstyle{l}$} (0.6,-1.1);
  \draw[postaction={decorate}]
    (-0.6,-1.1) to[out=90, in=270] node[near end, auto]{$\scriptstyle{k}$} (-1, -.1);
\end{scope}

\draw[<-] (0,-0.5) +(400:0.2) arc(400:10:0.2) node[right]{$\scriptstyle{l}$};

\draw[->] (-1.4,-1.1)  arc(0:90:0.3) node[above]{$\scriptstyle{l}$};
\draw[] (-1.4,-1.1)  arc(0:180:0.3);

\end{tikzpicture} \\
&= (-1)^{kl}\qbinom{k+l}{k} \begin{tikzpicture}[mystyle]

\draw [] (-0.7,0.1) rectangle (0.7,-0.1);

\begin{scope}[decoration={markings, mark=at position 0.5 with {\arrow{>}}}]
  \draw[postaction={decorate}]
    (0.2, 0.1) to[out=90, in=-90] node[near start, auto, swap]{$\scriptstyle{k}$} (0.0,1.1);
  \draw[postaction={decorate}]
    (-1,1.1) to[out=-90, in=90] node[near end, auto, swap]{$\scriptstyle{l}$} (-2,-1.1);
  \draw[postaction={decorate}]
    (-1.5,-1.1) to[out=90, in=90] node[near start, auto, swap]{$\scriptstyle{l}$} (-0.5, 0.1);

  \draw[postaction={decorate}]
    (0.0,-1.1) to[out=90, in=270] node[near end, auto]{$\scriptstyle{k}$} (-0.2, -.1);
  \draw[postaction={decorate}]
    (1,-1.1) to[out=90, in=-90] node[near end, auto, swap]{$\scriptstyle{l}$} (2,1.1);
  \draw[postaction={decorate}]
    (1.5,1.1) to[out=-90, in=-90] node[near start, auto, swap]{$\scriptstyle{l}$} (0.5, -0.1);

\end{scope}






\end{tikzpicture}\\
&= \iota_{-l} \otimes \iota_{k+l} \otimes \iota_{-l}\\
\end{align*}

The second equality holds by performing two multi-pop-switch relations:
one on the $\beta_{-l}$ at the top with the $l$ strands to the left and the $l$ strands on the bottom left,
and the other on the $\beta_l$ and the $l$ strands on the right and top right. 
Now expand the Jones-Wenzl idempotent.
The only non-zero term come from one of the following Temperley-Lieb diagrams.
$$\begin{tikzpicture}[mystyle]
\arc{0.7}{.4}{-0.5}{\scriptstyle{l}}{}
\arc{-0.7}{-.4}{0.5}{\scriptstyle{l}}{}
\draw[rounded corners=.25cm]
   (-0.5,-0.7) --(-0.5,-0.4) --(0.5,0.4)--(0.5,0.7) ;
\draw (0.9,.5) node  {$\scriptstyle{k-l}$};
\end{tikzpicture},
\qquad
\begin{tikzpicture}[mystyle]
\arc{0.7}{.4}{0.5}{\scriptstyle{k}}{}
\arc{-0.7}{-.4}{-0.5}{\scriptstyle{k}}{}
\draw[rounded corners=.25cm]
   (0.5,-0.7) --(0.5,-0.4) --(-0.5,0.4)--(-0.5,0.7) ;
\draw (-0.9,.5) node  {$\scriptstyle{l-k}$};
\end{tikzpicture}.
$$
Thus the result of $g\circ f$ must be a scalar times $\iota_{-l} \otimes \iota_{k+l} \otimes \iota_{-l}$.
Since $f \circ g$ is the identity
and $g \circ f$ is a scalar times the identity,
that scalar must be 1. 
\end{proof}

\section{Graph planar algebra and the Temperley-Lieb planar algebra}
\label{sec:gpa}

This section is motivation
for the definition of the pop-switch planar algebra.
We start with a summary of the definition of the graph planar algebra,
first defined in \cite{JonesGPA}.

Throughout this section,
fix a simple graph $\Gamma$.
For Jones, all planar algebras are shaded,
and $\Gamma$ is required to be bipartite.
We will ignore this issue.

Let $\mu$ be a function from the vertices of $\Gamma$ to $\mathbf{R}_{>0}$.
We will define the graph planar algebra $\mathcal{P}$
corresponding to $(\Gamma, \mu)$.

For each $k > 0$,
let $\mathcal{P}_{2k}$ be the vector space of complex valued functions
on the set of loops of length $2k$ on $\Gamma$.

Suppose $T$ is a tangle.
For each input disk of $T$,
let $v_b$ be a corresponding input vector.
We must define a corresponding output vector $v$.
Thus we must define
$v(\gamma)$ for every loop $\gamma$ in $\Gamma$
that has length equal to the number of endpoints on the outer boundary of $T$.

A {\em state} $\sigma$ of $T$ is a function
from the set of regions of $T$ to the set of vertices of $\Gamma$
such that adjacent regions are sent to adjacent vertices.

Suppose $r$ is a region of $T$.
This is a planar surface with boundary
that may include some right-angled corners.
The {\em Euler measure} $e(r)$
is defined in a similar way to the Euler characteristic,
using the usual formula $V - E + F$ for a triangulation of $r$.
The difference is,
every corner must be a vertex and only counts as $\frac{1}{4}$,
any other vertex on a boundary only counts as $\frac{1}{2}$,
and every edge on a boundary only counts as $\frac{1}{2}$.

We are finally ready to define the image vector $v$
of the vectors $v_b$
under the action of the tangle $T$.
$$
v(\gamma) =
\sum_\sigma
\left( \prod_r \mu(\sigma(r))^{e(r)} \right)
\left( \prod_b v_b(\sigma|_{\partial b}) \right).
$$
The sum is over all states $\sigma$ that are compatible with $\gamma$.
The first product is over all regions $r$ of $T$.
The second product is over all input disks $b$ of $T$.

The Temperley-Lieb planar algebra
is a subfactor planar algebra of type $A_\infty$.
It can be found inside
the graph planar algebra associated to $\Gamma = A_\infty$,
which is the ray with vertices indexed by positive integers.
The function $\mu$
assigns the quantum integer $[n]$ to the $n$th vertex.
(Note we are still assuming $q$ is not a root of unity.
If $q$ is a primitive $(n+1)$th root of unity
then we should use the graph $A_n$.)

Suppose $T$ is an oriented tangle.
Define a state of $T$
to be a function from the set of regions of $T$
to the set of vertices of $A_\infty$
such that,
for any strand of $T$,
if the region to its right is sent to vertex $n$
then the region to its left is sent to vertex $n+1$.
Thus,
a state is determined by the vertex associated to a single region.
In a sense,
the orientation on the strands
removes the ambiguity in the state of a Temperley-Lieb diagram.

Now suppose $T$ and $T'$ differ by a pop-switch relation.
There is an obvious correspondence between
states of $T$ and states off $T'$.
Furthermore,
the total Euler measure of the region associated to any given vertex is the same.
We therefore have
a well-defined embedding
of the pop-switch planar algebra
in the graph planar algebra of the graph $A_\infty$.

One can think of the pop-switch planar algebra
as a diagrammatic way to keep track of
computations inside the graph planar algebra of $A_\infty$.

\bibliographystyle{alpha}
\bibliography{jwProj}

\end{document}